\documentclass[11pt]{amsart}

\usepackage{epigamath}

\usepackage[english]{babel}

\numberwithin{equation}{section}

\usepackage{enumitem}
\usepackage{diagrams}

\newtheorem{Thrm}{Theorem}[section]
\newtheorem{Cor}[Thrm]{Corollary}
\newtheorem{Fact}[Thrm]{Fact}
\newtheorem{Lem}[Thrm]{Lemma}

\newtheorem*{Thrm*}{Theorem}

\theoremstyle{definition}
\newtheorem{Def}[Thrm]{Definition}

\theoremstyle{remark}
\newtheorem{Rem}[Thrm]{Remark}

\newcommand{\C}{\mathbb{C}}
\newcommand{\Gm}{\mathbb{G}_{m}}

\newcommand{\R}{\mathbb{R}}
\newcommand{\Z}{\mathbb{Z}}

\newcommand{\Ac}{\mathcal{A}}

\newcommand{\Fc}{\mathcal{F}}
\newcommand{\Gc}{\mathcal{G}}
\newcommand{\Hc}{\mathcal{H}}

\newcommand{\Kc}{\mathcal{K}}
\newcommand{\Lc}{\mathcal{L}}
\newcommand{\Mcal}{\mathcal{M}}
\newcommand{\Nc}{\mathcal{N}}
\newcommand{\Oc}{\mathcal{O}}

\newcommand{\Sc}{\mathcal{S}}
\newcommand{\Tc}{\mathcal{T}}
\newcommand{\Qc}{\mathcal{Q}}

\newcommand{\Xc}{\mathcal{X}}
\newcommand{\Yc}{\mathcal{Y}}

\newcommand{\ch}{\mathrm{ch}\,}
\newcommand{\chn}{\mathrm{ch}}
\newcommand{\cn}{\mathrm{c}}
\newcommand{\Coh}{\mathrm{Coh}\,}

\newcommand{\Lab}{\mathrm{L}^0}

\newcommand{\NumR}{\mathrm{Num}_{\R}\,}

\newcommand{\QCoh}{\mathrm{QCoh}\,}
\newcommand{\Rab}{\mathrm{R}^0}

\newcommand{\rk}{\mathrm{rk}\,}
\newcommand{\SD}{\mbox{\textbf{D}}}

\newcommand{\Spec}{\mathrm{Spec}\,}
\newcommand{\relSpec}{\mathcal{S}\mathrm{pec}}

\newcommand{\Chow}{\mathrm{A}^\bullet}
\newcommand{\K}{\mathrm{K}_0\,}
\newcommand{\Kany}{\mathrm{K}_0^*\,}
\newcommand{\Kor}{\mathrm{K}_0^{\mathrm{or}}\,}

\newcommand{\Kt}{\mathrm{K}_0^{\mathrm{t}}\,}

\newcommand{\Db}{\mathrm{D^b}}

\newcommand{\Hom}[2]{\mathrm{Hom}_{#1} \left( #2 \right)}
\newcommand{\Hrm}{\mathrm{H}}
\newcommand{\RGamma}{\mathrm{R}\Gamma}

\newcommand{\deq}{\overset{\text{def}}{=}}

\newcommand{\fprod}[1]{\underset{#1}{\times}}

\EpigaVolumeYear{5}{2021} \EpigaArticleNr{19} \ReceivedOn{February 15,
2021}
\InFinalFormOn{July 28, 2021}
\AcceptedOn{October 25, 2021}

\title{Moduli spaces of stable sheaves over quasi-polarized surfaces, and the relative Strange Duality morphism}
\titlemark{Moduli spaces of stable sheaves and the relative Strange Duality morphism}

\author{Svetlana Makarova}
\address{University of Pennsylvania, David Rittenhouse Laboratory, 209 South 33rd Street,
Philadelphia, PA 19104-6395}
\email{murmuno@yandex.ru}

\authormark{S.~Makarova}

\AbstractInEnglish{The main result of the present paper is a construction of relative moduli spaces of stable sheaves over the stack of quasi-polarized projective surfaces. For this, we use the theory of good moduli spaces, whose study was initiated by Alper.
As a corollary, we extend the relative Strange Duality morphism to the moduli space of quasi-polarized K3 surfaces.}



\TitleInFrench{Espaces des modules de faisceaux stables sur les surfaces quasi polaris\'ees et le morphisme de dualit\'e \'etrange relatif}

\AbstractInFrench{Le r\'esultat principal du pr\'esent article est une construction des espaces de modules relatifs des faisceaux stables au-dessus du champ des surfaces quasi polaris\'ees. Pour cela, nous utilisons la th\'eorie des bons espaces de modules, dont l'\'etude a \'et\'e initi\'ee par Alper. Comme corollaire, nous \'etendons le morphisme de dualit\'e \'etrange relatif \`a l'espace de modules des surfaces K3 quasi polaris\'ees.}

\begin{document}



\maketitle

\begin{prelims}

\DisplayAbstractInEnglish

\bigskip

\DisplayKeyWords

\medskip

\DisplayMSCclass

\bigskip

\languagesection{Fran\c{c}ais}

\bigskip

\DisplayTitleInFrench

\medskip

\DisplayAbstractInFrench

\end{prelims}


\newpage

\setcounter{tocdepth}{1}

\tableofcontents


\section{Introduction}

The work on the present paper started with an attempt to strengthen the results on Strange Duality on K3 surfaces, and is largely motivated by the approach of Marian and Oprea \cite{MO}.
Strange Duality is a conjectural duality between global sections of two natural line bundles on moduli spaces of stable sheaves.
It originated as a representation theoretic observation about pairs of affine Lie algebras, and then was reformulated geometrically over the moduli of bundles over curves \cite{DT}, \cite{Beau}.
In our paper, we develop the geometric approach 
to Strange Duality over surfaces
in the spirit of Marian and Oprea. They proved the Strange Duality conjecture for Hilbert sche\-mes of points on surfaces, 
moduli of sheaves on elliptic K3 surfaces with a section \cite{MO}; and cases for abelian surfaces \cite{MO14:abelian}, including a joint work \cite{BMOY:abelian}. The latter used birational isomorphisms of moduli spaces of stable sheaves with Hilbert schemes of points on the same K3 surface, following Bridgeland \cite{Br}, to reduce the question to the known case of Hilbert schemes. Further, Marian and Oprea use this result to conclude the Strange Duality isomorphism for a generic K3 surface in the moduli space of polarized K3 surfaces of degree at least four \cite{MO} (the idea first appeared in their earlier paper \cite{MO10}), for a pair of vectors whose determinants are equal to the polarization.

In order to make this argument work for K3 surfaces of degree two, we have to construct moduli spaces of stable sheaves over the stack of \emph{quasi}\-po\-la\-rized K3 surfaces, without assuming that the quasi-polarization (a choice of a big and nef line bundle) is ample. This is needed because elliptic K3 surfaces of degree two are not polarized, so the original approach of Marian and Oprea needs modification.
The question of whether the Strange Duality construction can be extended from the polarized locus to the whole moduli stack of quasi-polarized K3 surfaces was left open in \cite{MO}.
Stepping away from the ample locus requires 
that we retrace classical results in moduli theory: we prove openness of the stable locus, show that relative moduli spaces exist, and use the theory of good moduli spaces to derive gluing and descent results.
This notion was introduced by Jarod Alper \cite{Alp},
and further developed by Alper, Hall,  Halpern-Leistner, Heinloth and Rydh in numerous works; the most important for the present paper will be a recent remarkable result giving
a criterion for when a stack has a good moduli space \cite{AHHL}. 
This part of our work culminates in the following result, which we consider the main contribution of the paper:
\begin{Thrm}[Theorems~\ref{Thrm:good_mm} and~\ref{Thrm:construct_Mc}]
    Let $\Kc$ be the moduli stack of quasi-polarized projective surfaces, and let $\Xc$ be the universal surface with the universal quasi-polarization $\Hc$.
    Fix a Chern character $v$ over $\Xc$.
    Assume that,
    pointwise over $\Kc$,
    slope stability is equivalent to slope semistability for sheaves in class $v$. Then the stack of stable sheaves $\Qc \to \Kc$ of K-theory class $v$ is algebraic.
    Further, there exists a relative good moduli space $\Qc \to \Mcal$.
    The stack $\Mcal \to \Kc$ is fiberwise
    $($i.e. over each closed point of $\Kc)$
    the moduli scheme of stable sheaves of class $v$ with respect to the restriction of the universal quasi-polarization.
\end{Thrm}

Then we apply the developed theory to construct the Strange Duality morphism: this requires knowing that we have a good morphism to $\Mcal$ from the moduli stack which possesses a universal family of stable sheaves. Along the way, we use the Descent Lemma (Lemma \ref{Lem:good_descent}), where we show that quasi-coherent sheaves descend along good morphisms.

\begin{Thrm}[Equation \eqref{Eq:SD_mm}]
The Strange Duality morphism exists for a pair of orthogonal K-theory vectors on the universal K3 surface $\Xc \to \Kc$. It is defined up to a twist by a line bundle.
\end{Thrm}

\begin{Rem}
We attempted to use the Marian-Oprea trick (from their paper \cite{MO}) to extend the generic Strange Duality isomorphism to degree two.
Employing the relative moduli space construction from the present paper, it works as follows: working with an elliptic K3 surface (which in degree two lies in the quasi-polarized locus), use a Fourier-Mukai functor to establish a birational isomorphism of a pair of Hilbert schemes with a pair of moduli spaces of higher-rank sheaves, and using functoriality, identify the theta divisors on the two spaces; this proves the Strange Duality over the elliptic locus; by continuity, the Strange Duality morphism would be an isomorphism on a dense open substack. For this to work, one needs to find a pair of orthogonal vectors of rank at least three and a suitable Fourier-Mukai kernel to get to a pair of vectors of rank one whose sum of determinants is big and nef. The author could not find such vectors for the following choices of kernel: the ideal sheaf of the diagonal on the fibered square of the K3 surface, and a universal sheaf classifying rank $d+1$, degree $d$ stable fiber sheaves. The author is working on a more explicit description of other possible Fourier-Mukai kernels.
\end{Rem}

\subsubsection*{Outline of the paper}
We start with constructing relative moduli spaces in Section \ref{Sec:Rel_msp}. We first show that the stack of stable sheaves with respect to the universal quasi-polarization is algebraic. Then we recall some theory of good moduli spaces, and prove Descent Lemma (Lemma \ref{Lem:good_descent}) for good morphisms. Finally, we construct the relative space of stable sheaves locally over schematic charts of $\Kc$, and then glue the resulting spaces using their universal properties.  
Then, we apply the developed theory to the Strange Duality.
In Section \ref{Sec:SD_mm}, we start with generalizing Marian-Oprea's construction of the theta line bundles, and use it to extend the Strange Duality morphism to the quasi-polarized locus.

\subsubsection*{Conventions}
We work over an algebraically closed field of characteristic zero.
We write $(-)^\vee$ for the derived dual of a sheaf and $-\otimes -$ for the derived tensor product.
Given a morphism of schemes $f : X \to Y$, we denote by $f_*$ and $f^*$ the derived functors of pushforward and pullback, respectively. When we want to work with the classical functors instead of derived, we write $\Lab f^*$ for the nonderived pullback and $\Rab f_*$ for nonderived pushforward.
Note however that we distinguish between $\mathrm{Hom}$ and $\mathrm{RHom}$ (because $\mathrm{Hom}$ makes sense in the derived category on its own).

For the moduli theory of sheaves, whenever we say stability, we mean slope stability with respect to a chosen quasi-polarization.
We generally need a way to fix a numerical characteristic of the sheaves in question in order to obtain any finiteness results.
So, for a stack $\Xc$, we use zeroth algebraic K-theory $\K \Xc$ and zeroth topological K-theory $\Kt \Xc$ (defined by Blanc \cite{Blanc} for $\C$-stacks, and by Blanc, Robalo, To\"en, Vezzosi \cite{BRTV} in greater generality).
For a complex variety $X$, we can also define oriented topological K-theory $\Kor X$ by fixing the determinant of a topological K-theory vector. 
We will call a vector $v$ in any K-theory $\Kany X$ a fixed
\emph{numerical characteristic}, or
\emph{K-theory class}. When we need to be specific, we will add adjectives algebraic, topological or oriented topological to refer to the corresponding variants of K-theory.

Let $\Chow X$ denote the Chow ring of a smooth projective variety $X$.
It is well-known that there is a function called Chern character $\ch : \Db X \to \Chow X$, from objects of the derived category to the Chow ring, that factors as a ring homomorphism through the Grothendieck group: $\ch : \K X \to \Chow X$. Note that the Euler pairing descends to each of the K-groups by taking representative complexes $E$ and $F$ and computing Euler characteristic of their derived tensor product:
$$\chi(E \otimes F) \deq \chi \left( \RGamma(E \otimes F) \right)
\mbox{.}$$

We don't use Chern classes a lot, and instead we prefer to write a K-theory vector $v$ in terms of components of its Chern character: $\chn_0 v = \rk v$, $\chn_1 v = \cn_1 v$, $\chn_2 v = \frac{1}{2} (\cn_1 v)^2 - \cn_2 v$, etc.

\subsection*{Acknowledgements}
First and foremost, I would like to thank my adviser Davesh Maulik for suggesting the topic and insightful discussions throughout the work.
I have benefitted from discussions related to this work with
Alina Marian (my deepest appreciation for helping me escape deadends and commenting on the paper draft),
Dragos Oprea (for the comments on the paper draft),
Dmitry Makarov (for helping with a setup of a brute force algorithm),
Valery Alexeev,
Arend Bayer,
Dori Bejleri,
Tom Bridgeland,
Elden Elmanto, 
Nikon Kurnosov,
Emanuele Macr\`i,
Eyal Markman,
Tony Pantev,
David Rydh,
Evgeny Shinder,
Kota Yoshioka,
Xiaolei Zhao.
In addition, I am extremely grateful for the work of an anonymous referee who suggested many remarks that improved the clarity of exposition.

\section{Relative moduli spaces of stable sheaves with respect to a quasi-polarization}

\label{Sec:Rel_msp}

Let $\Kc$ be a stack of quasi-polarized projective surfaces that admits a universal family $u: \Xc \to \Kc$ with universal quasi-polarization $\Hc$. It means that we want $\Hom{}{T,\Kc}$ to classify families on a scheme $T$ given by pullbacks of $\Xc$ and $\Hc$ to $T$. Fix a K-theory class $v$ over $\Xc \to \Kc$.

For the main application, $\Kc$ will be the moduli stack of quasi-polarized K3 surfaces and  $u: \Xc \to \Kc$ will be the universal quasi-polarized K3 surface with quasi-polarization $\Hc$. Our aim is to define the relative moduli space of slope stable sheaves
$\Mcal \deq \Mcal_v \to \Kc$.

Our idea is to start with the moduli functor $\widetilde \Mcal$ of all flat families of sheaves of fixed K-theory class. Usually properness of support is assumed, but in our case it is an automatic condition due to the projectivity assumption. It is known that this functor is representable by an Artin stack, see for example a very general result of Lieblich (\cite{Lieb}, the main theorem). Then we will prove that the subfunctor $\Qc \subset \widetilde \Mcal$ of stable sheaves is open, hence also is an Artin stack. This is well-known when quasi-polarization is ample, but we will need additional technical arguments in order to generalize it to the non-polarized locus of K3 surfaces. 
In the next step, we observe that $\Qc$ admits a ``good moduli space morphism'' onto a relative moduli space, which will be denoted by $\Mcal = \Mcal_v$. For this, we use the result of existence of good moduli spaces by Alper, Heinloth and Halpern-Leistner \cite{AHHL}. Note that the fiber of $\Mcal$ is a scheme over each K3 surface $[X] \in \Kc$, but globally $\Mcal$ is still a stack.

Pointwise it is well-known that $\Mcal$ is a scheme for the polarized case.
However, to our knowledge,
the case of non-ample quasi-polarization,
and a construction of a good moduli space morphism are new results. They are summarized in the main theorem of this section:

\begin{Thrm*}[\emph{cf.} Theorem~\ref{Thrm:construct_Mc}]
    Let $\Kc$ be a stack of quasi-polarized surfaces that admits the universal surface $\Xc$ with the universal quasi-polarization $\Hc$.
    Fix a K-theory class $v$ over $\Xc$.
    Assume that,
    pointwise over $\Kc$,
    stability is equivalent to semistability for sheaves in class $v$.
    Then there exists a stack $\Mcal \to \Kc$ which is fiberwise
    $($\textit{i.e.} over each closed point of $\Kc)$
    the moduli scheme of stable sheaves of class $v$ with respect to the restriction of the universal quasi-polarization.
\end{Thrm*}

\subsection{Constructibility and generization}

We start with proving that the subfunctor $\Qc \subset \widetilde \Mcal$ is constructible and preserved by generization. Then, by a topological lemma, we will be able to conclude that it is an open subfunctor.

Recall that a stack is a contravariant (quasi-)functor $\Sc{}ch^{op} \to \Gc{}pd$ from the category of schemes $\Sc{}ch^{op}$ to the 2-category of groupoids $\Gc{}pd$ which satisfies a ``level two'' sheaf condition.

\begin{Def}
    Consider a moduli problem
    $\Fc : \Sc{}ch^{op} \to \Gc{}pd$ and a subfunctor $\Gc \subset \Fc$. We say that $\Gc$ is a constructible subfunctor of $\Fc$ if, for any  family $X \in \Fc(\tilde B)$ parametrized by the scheme $\tilde B$, the locus $B \deq \left\{
    b \in \tilde B \mid
    X_b \in \Gc(b)
    \right\}$
    is a constructible subset of $\tilde B$.
\end{Def}

\begin{Lem}
\label{constructibility}
    Fix a Chern character $v$ over $\Xc$.
    Assume that pointwise on $\Kc$, stability with respect to $\Hc$ is equivalent to semistability.
    Then the moduli subfunctor $\Qc \subset \widetilde \Mcal$ of stable sheaves is a constructible subfunctor. 
\end{Lem}

\begin{proof}
We will be checking constructibility by taking families of $\Qc$ and $\widetilde \Mcal$ parametrized by an arbitrary scheme $\tilde B$. Note that this condition can be checked on an open cover, so by possibly taking affine opens in $\tilde B$, we can assume that $\tilde B$ is quasi-compact and quasi-separated. 

Further, we can reduce the question to a Noetherian base by using Noetherian approximation by Thomason--Trobaugh \cite{TT} as follows. By \cite[Theorem~C.9]{TT}, a quasi-compact and quasi-separated scheme $\tilde B$ over a ground field admits an approximation by Noetherian schemes $C_i$; moreover, the bonding maps of the system are all affine:
$$  \tilde B = \lim C_i . $$

Both $\Qc$ and $\widetilde \Mcal$ evaluated at $\tilde B$ parametrize certain sheaves over $X \deq \Xc \fprod{\Kc} \tilde B$,
which, being a family of K3 surfaces, is a scheme. By \cite[Lemma~01ZM]{Stacks}, we can choose a Noetherian $X_i$ and $X_i \to C_i$ for sufficiently large $i$ such that
$X \cong \tilde B \fprod{C_i} X_i$. So, possibly taking a subset of the indexing set, we can assume that $X_i$ is chosen for all $i$.
Then by a simple category theory fact, we have $X \cong \lim X_i$. 
By \cite[Lemmas~0B8W and~05LY]{Stacks}, a flat sheaf on $X$ is a pullback of some flat sheaf on a finite step $X_i$. 
Therefore, we can study a particular flat family parametrized by a finite step $C_i$, which means that without loss of generality, we can assume that $\tilde B$ is Noetherian.

Finally, note that both constructibility and stability can be checked on closed points, so it is enough to check the condition for reduced Noetherian bases.

Let $F$ be a family of sheaves parametrized by a reduced Noetherian base $\tilde B$, that is $F \in \Coh X$ is a coherent sheaf over a family of quasi-polarized K3 surfaces $X \to \tilde B$ of Chern class $v$ and flat over $\tilde B$.
We want to show that the locus
$B \deq \left\{ b \in \tilde B \mid
F_b \mbox{ is stable} \right\}$ is constructible.

To that end, denote by $H$ the quasi-polarization of $X \to \tilde B$.
We will use Noetherian induction on the base: we will stratify $\tilde B$ with locally closed disjoint subsets $B_i$, and prove that $B \cap B_i$ is open in each $B_i$. The Noetherian property will be used to prove that the set $\left\{ B_i \right\}$ of the strata is finite.

Note that the locus $B_0$, where the quasi-polarization $H_{X_b}$ is ample, is open.
It is a standard result that semistability is open in flat families \cite[Proposition~2.3.1]{HuL}; with our assumption that semistability implies stability, we then obtain that the stable locus $B_0 \cap B$ is open in $B_0$.

Consider the strictly quasi-polarized locus $\tilde B \setminus B_0$ (\emph{i.e.} where the quasi-polarization is not ample) and pick an irreducible component $B_1$ with the generic point $\eta$. The surface $X_\eta$ is projective, so we can pick an ample line bundle $L_\eta$ over $X_\eta$. Note that $B_1$ and the restriction $X_{B_1}$ are integral schemes, hence the sheaf of total quotient rings of $\Oc_X$ is the constant sheaf $K$ with fiber equal to the field of fractions of the generic point $\kappa \in X_\eta \subset X$, and so every line bundle on $X_\eta$ comes from a Cartier divisor.
Let's say that $L_\eta \cong \Oc_{X_\eta}(D_\eta)$ for $D_\eta \in \Hrm^0 \left( X_\eta, K_{X_\eta}^\times / \Oc_{X_\eta}^\times \right)$.
We can extend this divisor to a Cartier divisor $D$ over some open subset $U'$ of $X$.

Let us for a moment denote by $f$ the morphism $X \to B$.
We now argue that $U'$ can be extended to an open set of the form $f^{-1}(U)$ for some $U \subset B$.
The morphism $f$ is a flat family of projective surfaces, so by \cite[Lemma~01UA]{Stacks}, it is open. So the set $U = f(U') \subset B$ is open.
Since every fiber is proper, one can note that every regular function is constant along a fixed fiber. Therefore, if the section corresponding to $D$ is defined at one point of a fiber, it is defined over the whole fiber. So $D_\eta$ can be extended to a divisor on $f^{-1}(U)$, and we get an extension of $L_\eta$ to $L$ over $X_U$.

Note further that being ample is an open condition, so we may assume, after possibly shrinking $U$, that $L$ is relatively ample. Now we will show that we can pick a small enough $\epsilon \in \R^{+}$ such that stability with respect to $H_U + \epsilon L$ is equivalent to stability with respect to $H_U$ for every point in $U$. The argument for local finiteness of the walls \cite[Lemma~4.C.2]{HuL} (the result is summarized in \ref{hyparr}) can be extended to a neighborhood, and on each of those, we pick $\epsilon$ as above; then by quasi-compactness of the base $U$ we can pick the minimum of the $\epsilon$'s for every open neighborhood. Now we have a polarization over $X_U$ and can deduce openness of the locus where $F_U$ is stable on a fiber with respect to the ample $H_U + \epsilon L$. This locus is exactly $B \cap U \subset B_1$.

At this point, we want to redefine $B_1$ to be $U$, and pass to consideration of the closed subset $\tilde B \setminus (B_0 \sqcup B_1)$ of $\tilde B$. The choice of the subsequent $B_i$'s is done in the same fashion.

By the Noetherian assumption, there are only finitely many $B_i$'s, and for each of those, the subset $B \cap B_i$ is open inside $B_i$. Since $B_i$ is locally closed inside $\tilde B$, we get that $B$ is equal to the finite disjoint union of locally closed subsets $B \cap B_i \subset \tilde B$, hence constructible.
\end{proof}

\begin{Def}
    Consider a moduli problem
    $\Fc : \Sc{}ch \to \Gc{}pd$ and a subfunctor $\Gc \subset \Fc$. We say that $\Gc$ is closed under generization if, for any family $X \in \Fc(\Spec R)$ parametrized by the spectrum of a valuation ring $R$
    such that the fiber of $X$ over the closed point belongs to the subfunctor $\Gc$, the generic fiber is also in $\Gc$.
\end{Def}

\begin{Lem}
\label{generization}
    Fix a Chern character $v$ over $\Xc$.
    Then the moduli subfunctor $\Qc \subset \widetilde \Mcal$ of stable sheaves is preserved under generization.
\end{Lem}

\begin{proof}
    Assume that $F$ is a flat family over $X$ parametrized by $\Spec R$, where $R$ is a valuation ring with fraction field $K$ and residue field $k$. Assume further that $F$ is stable when restricted to the closed fiber $X_k$. We want to prove that its restriction $F_K$ to the open fiber is also stable.
    
    To that end, pick a proper quotient sheaf $F_K \to G_K \to 0$.
    
    We consider slope stability with respect to a quasi-polarization $H$
    which may not be ample. The function
    $P(n) \deq P_{G_K}^{H_K}(n) \deq \chi(X_K, G_K \otimes H_K^n)$
    is not usually called Hilbert polynomial for just big and nef line bundles, so we will call it a quasi-Hilbert polynomial. Note that this function is still a polynomial, because the standard argument still applies.
    
    Since the Quot scheme $\mathrm{Quot}_{X/R}^{F,P}$ is proper, we can extend the quotient $F_K \to G_K \to 0$ to a flat quotient $F \to G \to 0$ over $X$, by the existence part of the valuative criterion. Recall that the slope is a rational function of some coefficients of the quasi-Hilbert polynomial, therefore it is constant in flat families, so $\mu_{H_K}(G_K) = \mu_{H_k}(G_k) < \mu_{H_k}(F_k) = \mu_{H_K}(F_K)$. And so we can conclude stability of $F_K$.
\end{proof}

\subsection{Technicalities to prove openness}

In this part, we will briefly remind a topology result that connects the properties of being constructible and open following Stacks Project \cite{Stacks}.

\begin{Def}[{\cite[Definition~004X]{Stacks}}]
\label{char_constructible}
    A topological space $X$ is called \emph{sober} if every irreducible closed subset has a unique generic point.
\end{Def}

\begin{Lem}[{\cite[Lemma~0542]{Stacks}}]
\label{Lem:Stacks_openness}
    Let $X$ be a Noetherian sober topological space. Let $E\subset X$ be a subset of $X$.
    If $E$ is constructible and stable under generization, then $E$ is open.
\end{Lem}

\begin{Cor}
\label{Lem:openness}
    Let $X$ be a Noetherian scheme and $E \subset X$ a constructible subset which is preserved by generization. Then $E$ is open in X.
\end{Cor}

\begin{proof}
We observe that the topological space of a Noetherian scheme is Noetherian sober. Then we apply Lemma \ref{Lem:Stacks_openness}.
\end{proof}

\subsection{Good morphisms}

In this subsection we remind the definition of a \emph{good morphism}, as introduced by Alper.
Our plan is to first prove that $\Qc$ -- the stack of stable flat sheaves of a fixed K-theory class -- admits a good moduli space when pulled back to a scheme; for this, we will heavily cite the work of Alper, Heinloth and Halpern-Leistner on existence of good moduli spaces \cite{AHHL}.
Then we show that these glue to a ``relative good moduli space'' $\Qc \to \Mcal$.

\begin{Def}
We now recall Alper's \cite{Alp} definition of good morphisms.
\begin{enumerate}[label = (\roman*)]
    \item
    Let $\Xc$ and $\Yc$ be two Artin stacks with a quasi-compact morphism $f : \Xc \to \Yc$. We call $f$ a \emph{good morphism} if
    $\Rab f_* : \QCoh \Xc \to \QCoh \Yc$ is exact and the natural map
    $\Oc_\Yc \to \Rab f_* \Oc_\Xc$ is an isomorphism.
    \item
    If in the previous definition $\Yc = Y$ is an algebraic space, then $f : \Xc \to Y$ is called a \emph{good moduli space}.
\end{enumerate}
\end{Def}

\begin{Rem}
    We adopt a shorter terminology ``good morphism'', while Alper calls that a ``good moduli space morphism''. Our choice is motivated by the belief that the notion of a good morphism is more fundamental than its application to moduli theory. One argument in support of this point of view is that good morphisms satisfy descent (for purely formal reasons), as we show now in Lemma \ref{Lem:good_descent}. We shall use the lemma later.
\end{Rem}

Let us recall a fundamental result about universality of good moduli spaces.

\begin{Def}
Given a morphism of stacks $\nu : \Qc \to \Nc$, denote the natural projections by
$\nu_i : \Qc \fprod{\Nc} \Qc \to \Qc$
and $\nu_{ij} : \Qc \fprod{\Nc} \Qc \fprod{\Nc} \Qc \to \Qc \fprod{\Nc} \Qc$.
We define the objects of the \emph{category of descent data}
$\QCoh (\nu : \Qc \to \Nc)$
as tuples $(F,\iota)$, with $F \in \QCoh(\Qc)$ and an isomorphism $\iota : \nu_1^* F \to \nu_2^* F$ subject to the \emph{cocycle condition}
$\nu_{13}^* \iota =
    \nu_{23}^* \iota \circ
    \nu_{12}^* \iota$.
Morphisms are those morphisms of quasi-coherent sheaves which commute with $\iota$.
\end{Def}

\begin{Def}
We say that $\QCoh$ \emph{satisfies descent along a morphism}
$\nu : \Qc \to \Nc$ if the functor given by 
$\Lab \nu^* : \QCoh \Nc \to \QCoh (\nu : \Qc \to \Nc)$ is an equivalence of categories.
\end{Def}

\begin{Lem}[Descent Lemma]
\label{Lem:good_descent}
    quasi-coherent sheaves satisfy descent along good morphisms.
\end{Lem}

\begin{proof}
Take a good morphism $\nu : \Qc \to \Nc$. We will argue that the functor $\Lab \nu^* : \QCoh \Nc \to \QCoh (\Qc \to \Nc)$ establishes an equivalence. Note that we have a right adjoint functor $\Rab \nu_* : \QCoh \Qc \to \Nc$. In the setup of having two adjoint functors, it is enough to prove that $\Lab \nu^*$ is fully faithful and the exact $\Rab \nu_*$ ``detects zero objects'' in $\QCoh (\Qc \to \Nc)$. At a glance, it is not obvious that $\Rab \nu_*$ should be a quasi-inverse, but being a good moduli space morphism (term introduced by Alper in his thesis paper \cite{Alp}) is a strong condition, so it will follow from the proof.
    
    First we will prove that $\Lab \nu^*$ is fully faithful. So consider the moprhism
    $$  \Lab \nu^* : \Hom{\Nc}{F,G} \to \Hom{\Qc}{\Lab \nu^* F, \Lab \nu^* G}
    \mbox{.} $$
    Note that by adjunction, the right hand side is isomorphic to:
    $$  \Hom{\Qc}{\Lab \nu^* F, \Lab \nu^* G} =
    \Hom{\Nc}{F,
    \Rab \nu_* (\Oc_\Qc \otimes \Lab \nu^* G)}
    \mbox{.} $$
    Further, Alper proved projection formula that is applicable in this setting, see his Proposition 4.5 together with Remark 4.4 in his paper \cite{Alp}, so we in fact can simplify the right hand side and get a morphism:
    $$  \Lab \nu^* : \Hom{\Nc}{F,G} \to \Hom{\Nc}{F,
    \Rab \nu_* (\Oc_\Qc) \otimes G}
    \mbox{.} $$
    But since by assuption we have $\Rab \nu_* (\Oc_\Qc) \cong \Oc_\Nc$, we get that $\Rab \nu^*$ induces an isomorphism on Hom-spaces, as desired. In particular, it follows that $\Rab \nu_* \Lab \nu^*$ is isomorphic to the identity functor.
    
Now we prove that $\Rab \nu_*$  ``detects zero objects''. Let $(G,\iota) \in \QCoh (\Qc \to \Nc)$ and assume that $\Rab \nu_* G = 0$. If $q_1$ and $q_2$ are two projections $\Qc \fprod{\Nc} \Qc \to \Qc$, then the gluing data $\iota$ is a fixed isomorphism $\iota : \Lab q_1^* G \to \Lab q_2^* G$. Now we would like to apply Alper's base change formula for good moduli space morphisms (see Lemma~4.7(iii) together with Remark~4.4 in \cite{Alp})
    to $\Rab q_{1*} \iota$ to get an isomorphism:
    $$ G \cong \Rab q_{1*}\Lab q_1^* G
    \cong \Rab q_{1*}\Lab q_2^* G \cong
    \Lab \nu^* \Rab \nu_* G = 0
    \mbox{.}$$
    
    To conclude that $\Lab \nu^*$ establishes an equivalence, we can formally observe that it is essentially surjective. Indeed, take some object $F \in \QCoh (\Qc \to \Nc)$ and consider the natural morphism:
    $$  \Lab \nu^* \Rab \nu_* F \to F
    \mbox{.} $$
    Now we can complete the sequence by kernel and cokernel and apply $\Rab \nu_* F$, which is exact, to the resulting sequence. From the above results, the middle morphism is an isomorphism, and since $\Rab \nu_* F$ detects zero objects, we conclude that both kernel and cokernel vanish. Therefore we conclude that $\Lab \nu^* \Rab \nu_* F \cong F$ and thus lies in the essential image of $\Lab \nu^*$.
\end{proof}

Returning to our situation, assume that we consider a pullback of $\Qc \to \Kc$ to any Noetherian affine scheme $K \to \Kc$, so we get an Artin stack over a Noetherian base $\Qc_K \to K$ that parametrizes flat sheaves with a fixed K-theory class over the family of quasi-polarized surfaces
$X \deq \Xc \fprod{\Kc} K \to K$.

\begin{Thrm}[{\cite[Theorem~6.6]{Alp}}]
\label{Thrm:univgms}
Suppose $\Xc$ is a locally Noetherian Artin stack and $f: \Xc \to Y $ a good moduli space. Then $f$ is universal for maps to algebraic spaces, i.e. for any algebraic space $Z$, the following natural map of sets is a bijection:
$$ f^* : \Hom{}{Y,Z} \to \Hom{}{\Xc,Z}
.$$
\end{Thrm}

\begin{Lem}
\label{Lem:good_moduli}
    Take a Noetherian scheme $K$ with a morphism $K \to \Kc$.
    Then the morphism $\Qc_K \to K$ admits a good moduli space.
\end{Lem}

\begin{proof}
We want to apply the criterion for existence of good moduli spaces (Theorem A in \cite{AHHL}), and so we check that the conditions in the criterion are verified.

By \cite[Example 7.1]{AHHL}, this stack $\Qc_K$ coincides with the moduli functor given by \cite[Definition~7.8]{AHHL}. Therefore by  \cite[Lemma~7.16]{AHHL}, this stack is $\Theta$-reductive (\emph{cf.} Definition~3.10 in \cite{AHHL}). The stabilizer groups of the stack $\Qc$ are all $\mathbb{G}_m$ by stability of sheaves, hence connected and reductive. So $\Qc$ is locally linearly reductive (\emph{cf.} Definition~2.1 in \cite{AHHL}), and by \cite[Proposition~3.56]{AHHL} it has unpunctured inertia (\emph{cf.} Definition~3.53 in \cite{AHHL}). By \cite[Lemmas~0DPW and~0DPX]{Stacks}, the stack $\Qc_K$ is of finite presentation and with affine diagonal.

So we can apply Theorem A of \cite{AHHL} to conclude that $\Qc_K$ admits a good moduli space $\nu_K : \Qc_K \to M_K$, and the morphism $\nu_K$ is universal for maps to an algebraic space by Theorem \ref{Thrm:univgms} (\cite[Theorem~6.6]{Alp}).
\end{proof}

We now want to show that the good moduli spaces $M_K \to K$ ``glue'' to a relative good moduli space $\Mcal \to \Kc$, that is there exists a good moduli space morphism $\nu : \Qc \to \Mcal$ such that $\Mcal \to \Kc$ is a relative algebraic space.

\begin{Thrm}
\label{Thrm:good_mm}
    There exists a relative good moduli space $\Mcal \to \Kc$ such that $\Mcal$ is an algebraic stack; for each scheme $K \to \Kc$, the pullback $\Mcal_K$ is isomorphic to $M_K$; and there exists a morphism $\nu : \Qc \to \Mcal$ which is good.
\end{Thrm}

\begin{proof}
Since the moduli stack of quasi-polarized K3 surfaces $\Kc$ is an Artin stack, we can choose a smooth surjection $K \to \Kc$ from a scheme $K$. This morphism is representable by algebraic spaces, so
the fibered product $K'\deq K \times_{\Kc} K$ is an algebraic space; and the projection morphisms $k_1, k_2: K' \rightrightarrows K$ are still smooth, being pullbacks of a smooth morphism.
The spaces $K$ and $K'$ naturally assemble into a smooth groupoid of algebraic spaces \cite[Lemma~04T4]{Stacks}, and the quotient groupoid is isomorphic to the original stack $\Kc \cong [K/K']$ \cite[Lemma~04T5]{Stacks}, so we have obtained a groupoid presentation of $\Kc$.

By Lemma \ref{Lem:good_moduli}, there exists a good moduli space $\Qc_K \to M_K$. Since good moduli spaces are universal for morphisms to algebraic spaces (Theorem \ref{Thrm:univgms}, \cite[Theorem~6.6]{Alp}), we also obtain the unique canonical morphism $u: M_K \to K$.

We now want to produce an algebraic space $P$ so that $P\rightrightarrows M_K$ becomes a smooth groupoid which would then yield a quotient stack. To that end, study the pullback
$$ P \deq K' \fprod{k_i,K} M_K =
(K \fprod{\Kc} K) \fprod{k_i,K} M_K = 
K \fprod{\Kc} M_K
.$$
The object $P$ does not depend (up to isomorphism) on the projection $k_i$ we choose, but the two projections induce two smooth morphisms $p_1, p_2: P\rightrightarrows M_K$, where $p_i = k_i \times 1_{M_K}$. Further, the rest of the structure maps for $P\rightrightarrows M_K$ --- composition, identity, inverse as in \cite[\S 0230]{Stacks} --- are obtained from the groupoid $K' \rightrightarrows K$ by pullback and yield the structure of a groupoid in algebraic spaces for $P \rightrightarrows M_K$ \cite[044B]{Stacks}.



Now, it is known that the quotient stack of a smooth groupoid is algebraic \cite[Theorem~04TK]{Stacks}, so we put $\Mcal \deq [M_K/P]$ to get the relative good moduli space. Since we had a morphism of groupoids
$$\Big[P \rightrightarrows M_K\Big] \to 
\Big[ K' \rightrightarrows K \Big],$$
we also obtain a morphism of the quotient stacks $\Mcal \to \Kc$ \cite[Lemma~046Q]{Stacks}.

To argue that we have a canonical morphism $\nu: \Qc \to \Mcal$, we will construct a morphism from a groupoid associated to $\Qc$ to the groupoid $P \rightrightarrows M_K$. Pick a smooth cover by a scheme $Q \to \Qc_K$ -- it induces a smooth cover $Q \to \Qc$. Denote by $v :Q \to M_K$ the composition of the cover with $\Qc_K \to M_K$.
Put $Q' = Q \times_\Qc Q$, then we get a groupoid presentation $q_1,q_2 : Q' \rightrightarrows Q$ of $\Qc$.  Let us summarize the notation in the diagram:
\begin{diagram}
Q' & 
    \pile{\rTo^{q_1,q_2} \\ \rTo} &
    Q & \rTo & \Qc_K & \rTo^q & \Qc \\
\dDashto && \dTo_v &
    \ldTo^g &&& \dDashto_\nu \\
P & 
    \pile{\rTo^{p_1,p_2} \\ \rTo} &
    M_K && \rTo^p && \Mcal \\
\dTo && \dTo_u && && \dTo \\
K' & 
    \pile{\rTo^{k_1,k_2} \\ \rTo} &
    K && \rTo && \Kc \\
\end{diagram}
Since $K' = K\times_\Kc K$, the two morphisms $uvq_i : Q' \rightrightarrows K$ define a canonical morphism $w: Q' \to K'$. Then the pair of morphisms $(w,vq_i)$ for any $i=1,2$ define a canonical morphism to the fibered product $Q' \to K' \times_K M_K = P$, and we then have a morphism of groupoids
$$\Big[ Q' \rightrightarrows Q \Big] \to 
\Big[ P \rightrightarrows M_K \Big]$$
which induces a morphism of the quotient stacks $\nu : \Qc \to \Mcal$. 

We can now check that $\nu$ is good. First, let us study $\Rab \nu_* \Oc_\Qc$. By descent, it is isomorphic to $\Oc_\Mcal$ if and only if its pullback $\Lab p^* \Rab \nu_* \Oc_\Mcal$ is isomorphic to $\Oc_{M_K}$. But $p$ is smooth, hence flat, so by base change \cite[Corollary~1.4.(2)]{Hall:basechange}, and using that $g$ is a good moduli space,
we have:
$$ \Lab p^* \Rab \nu_* \Oc_\Mcal =
\Rab q_* \Lab q^* \Oc_\Mcal \cong \Oc_{M_K}
.$$
Using base change again, we can check that $\nu_*$ is exact, so $\nu$ is good.
\end{proof}

\begin{Rem}
\label{Rem:pointwise}
    The property of being a good moduli space is preserved under arbitrary base change \cite{Alp}, therefore, for a closed point $[X] \in \Kc$, the spaces $\Mcal_{[X]}$ and $M_{[X]}$ are isomorphic, so $\Mcal_{[X]}$ is a good moduli space of the stack of stable sheaves over the surface $X$.
\end{Rem}

\subsection{The good morphism is fiberwise a scheme}

We will briefly summarize several results about change of polarization from the book by Huybrechts and Lehn \cite[\S4.3]{HuL}. Then we will apply these results to our situation to show that for a closed point $[X] \in \Kc$, the fiber $\Mcal_{[X]}$ is a scheme. 

\begin{Fact}[\emph{cf.} {\cite[Lemma~4.C.2 and Theorem~4.C.3]{HuL}}]
\label{hyparr}
    Let $X$ be a smooth projective surface over an algebraically closed field of characteristic zero.
    For a fixed Chern character $v$ on $X$, there is a locally finite hyperplane arrangement (the hyperplanes are called \emph{walls}) in the numerical group $\NumR X$ satisfying the following property:
    if a big and nef divisor $H \in \NumR X$ is not on a wall
    and $\rk v$ is coprime with $\cn_1 (v)\cdot H$, then a torsion-free sheaf of Chern character $v$ is $H$-stable iff it is $H$-semistable.
\end{Fact}

\begin{Lem}
\label{schemeness}
    Fix a Chern character $v$ over $\Xc$.
    Assume that semistable sheaves of class $v$ are stable.
    Then for any quasi-polarized surface $[X,H] \in \Kc$, the restriction $\Mcal_{[X]}$ is a scheme.
\end{Lem}

\begin{proof}
    This is well-known in the case when the quasi-polarization is ample and follows from Remark \ref{Rem:pointwise} and the assumption that semistability is equivalent to stability. So we will reduce the general case $[X,H]$ with $H$ big and nef to the ample case by considering a small ample shift.
    
    For a big and nef $H$ (which may lie on a wall -- it wouldn't pose problems), we can find an ample divisor $H_1 \in \NumR X$ such that the semiopen line segment $(H,H_1]$ does not intersect any walls -- this follows from the fact that the hyperplane arrangement is locally finite (Fact \ref{hyparr}).
    From the assumption that $\cn_1(v)$ is indivisible and the same Fact \ref{hyparr} it also follows that stability with respect to any $H_\epsilon \in (H,H_1]$ is equivalent to semistability, and in addition,
    we assumed equivalence of $H$-stability and $H$-semistability.
    We now want to prove that in this setup, a sheaf $F$ is $H$-stable iff it is $H_1$-stable.
    
    Assume that it is $H$-stable, but not $H_1$-stable. Fix an $H_1$-destabilizing subsheaf $F_1 \subset F$ and let us define
    $\delta \deq
    \frac{\cn_1(F_1)}{\rk F_1} - 
    \frac{\cn_1(F)}{\rk F}$. 
    Note that pairing with $\delta$ is a linear function on $\NumR X$ and $H\cdot\delta < 0$ from $H$-stability of $F$, while $H_1\cdot\delta > 0$ from $H_1$-instability.
    Hence there exists some $H_\epsilon \in (H,H_1)$ such that $H_\epsilon\cdot\delta = 0$ proving that $F$ is strictly semistable with respect to $H_\epsilon$ and with destabilizig subsheaf $F_1$. But this contradicts our setup where stability is equivalent to semistability.
    
    The proof that $H_1$-stability implies $H$-stability is analogous.
\end{proof}

\begin{Rem}
    It is interesting to note that under the assumptions of the above lemma, the resulting moduli space with respect to quasi-polarization does not depend on the small ample shift, even if the two ample shifts are separated by a wall. The latter may happen when $H$ happens to be on a wall.
\end{Rem}

\subsection{Proof of the main theorem}

Now we can combine the above results and prove the following theorem.

\begin{Thrm}
\label{Thrm:construct_Mc}
    Let $\Kc$ be a stack of quasi-polarized  surfaces that admits the universal family $\Xc$ with the universal quasi-polarization $\Hc$.
    Fix a Chern character $v$ over $\Xc$.
    Assume that,
    pointwise over $\Kc$,
    stability is equivalent to semistability for sheaves in class $v$.
    Then there exists a stack $\Mcal \to \Kc$ which is fiberwise
    $($i.e. over each closed point of $\Kc)$
    the moduli scheme of stable sheaves of class $v$ with respect to the restriction of the universal quasi-polarization.
\end{Thrm}

\begin{proof}
We have proved in Lemmas \ref{constructibility} and \ref{generization} that $\Qc$ is constructible in $\widetilde \Mcal$ and preserved by generization.
Therefore, by Lemma \ref{Lem:openness}, the subfunctor $\Qc \subset \widetilde \Mcal$ is open, and since $\widetilde \Mcal$ is an Artin stack, then $\Qc$ is also an Artin stack.
By Theorem \ref{Thrm:good_mm}, there exists a good moduli space morphism $\nu : \Qc \to \Mcal$ such that fiberwise we get good moduli spaces. By Lemma \ref{schemeness}, the family $\Mcal \to \Kc$ is fiberwise a scheme.
\end{proof}
    
\section{The Strange Duality morphism}

\label{Sec:SD_mm}

\subsection{Defining theta line bundles}

Let $\Kc$ be the moduli stack of quasi-polarized K3 surfaces. Let $\Xc \to \Kc$ be the universal quasi-polarized K3 surface, so it will also be the moduli stack of pointed quasi-polarized K3 surfaces. Then $\Yc \deq \Xc \times_\Kc \Xc$ is the universal pointed quasi-polarized K3 surface.

For a fixed Chern character $v$, let $\Mcal_v \to \Kc$ be the stack of stable sheaves with Chern character $v$ which is pointwise a scheme, as constructed in Theorem \ref{Thrm:good_mm} and Theorem \ref{Thrm:construct_Mc}. So $\Mcal_v$ is a stack, but over each point of $\Kc$, the fiber is a scheme -- the moduli scheme of stable sheaves with respect to the quasi-polarization at this point of the moduli space.

Consider $\Nc_v \deq \Mcal_v \times_\Kc \Xc$ -- it will be the relative moduli space over the stack of pointed K3 surfaces. Unfortunately, there is no universal family over $\Nc_v$, so we need to work with the stack $\Qc \to \Nc_v$, which is the moduli stack of stable sheaves before we ``forget'' the $\mathbb{G}_m$-automorphisms of the sheaves. We can construct it analogously to Theorem \ref{Thrm:good_mm} or pull back the $\Qc$ from $\Qc \to \Kc$ along $\Xc \to \Kc$.
Then we have the universal family $E \in \Coh \left( \Yc \fprod{\Kc} \Yc \fprod{\Xc} \Qc \right)$ of stable sheaves.

Consider the following Cartesian square. We will use it to define a line bundle on $\Qc$ and, with Lemma~\ref{Lem:good_descent}, argue that it descends to $\Nc_v$, so that we can later use this universal theta line bundle to construct the Strange Duality morphism in families. 
\begin{diagram}
\Yc && \lTo^q    && \Yc \fprod{\Xc} \Qc
    \\
\dTo &&&& \dTo_p
    \\
\Xc & \lTo  & \Nc_v & \lTo  & \Qc
    \\
\end{diagram}

Taking an algebraic K-theory class $w$ on $\Yc$, we can use Fourier-Mukai transform and define uniquely up to an isomorphism a line bundle 
$$L \deq \det p_* (E \otimes q^* w)$$
on $\Qc$. Further, assuming that $w$ is orthogonal to $v$, we can argue that this line bundle $L$ descends along $\Qc \to \Nc_v$, as described in Lemma \ref{Lem:theta_on_ptd}. We will need the following preliminary result.

\begin{Lem}
\label{Lem:twisted_descent_datum}
Let $B$ be a locally Noetherian scheme and $\pi: E \to B$ be a $\Gm$-bundle over $B$, i.e. there is a line bundle $\Lc$ on $B$ such that $E = \relSpec_B \big( \bigoplus_{n \in \Z} \Lc^{\otimes n} \big)$. Let $F$ and $G$ be two indecomposable complexes of coherent sheaves on $B$ and assume that $\pi^* F \cong \pi^* G$. Then there exists $k\in \Z$ such that $F \cong G \otimes \Lc^k$.
\end{Lem}

\begin{proof}
Since coherent sheaves on a relative spectrum of a sheaf of algebras $\Ac = \bigoplus_{n \in \Z} \Lc^{\otimes n}$ correspond to quasi-coherent sheaves on the base $B$ that are finitely generated $\Ac$-modules, we can view the isomorphism
$\pi^* F \cong \pi^* G$ as an isomorphism of complexes of quasi-coherent sheaves on $B$:
$$ \bigoplus_{n \in \Z} \Lc^{\otimes n} \otimes F
\cong \bigoplus_{n \in \Z} \Lc^{\otimes n} \otimes G
.$$
Consider the direct summand $F = \Lc^{\otimes 0} \otimes F$ of the left hand side of the isomorphism.
$$  F \subset \bigoplus_{n \in \Z} \Lc^{\otimes n} \otimes F \xrightarrow{\cong} 
\bigoplus_{n \in \Z} \Lc^{\otimes n} \otimes G
.$$
Viewing $F$ as a subobject of the right hand side, we get a decomposition of $F$ into direct summands $F \cap \Lc^{\otimes n} \otimes G$; by assumption, a nontrivial decomposition cannot happen, so there is only one index $k$ for which $F \cap \Lc^{\otimes k} \otimes G \neq 0$, and therefore the morphism from $F$ factors through $\Lc^{\otimes k} \otimes G$. Using a similar argument for $\Lc^{\otimes k} \otimes G$, we can deduce that in fact $F$ is identified with $\Lc^{\otimes k} \otimes G$ by the isomorphism of pullbacks.
\end{proof}

\begin{Lem}
\label{Lem:theta_on_ptd}
    Take two orthogonal algebraic K-theory classes $v$ and $w$ on the universal pointed K3 surface $\Yc$. As before, $E \in \Coh \left( \Yc \fprod{\Xc} \Qc \right)$ is the universal family. Define $L \deq \det p_* (E \otimes q^* w)$ on $\Qc$. Then the line bundle $L$ descends to $\Nc_v$.
\end{Lem}

\begin{proof}
We will proceed as follows: first, we prove that the rank of $p_* (E \otimes q^* w)$ is zero using orthogonality of $v$ and $w$, then we recall that there exists a ``descent datum'' for $E$ which does not satisfy the cocycle condition, and we use it to construct descent datum for $L$, and finally we argue that the descent datum for $L$ satisfies the cocycle condition with the use of the first observation about rank.

\textit{Step 1: rank equals zero.}
Now we want to use orthogonality of $v$ and $w$ to prove that $\rk p_* (E \otimes q^* w) = 0$. For that, let us consider the restriction of this sheaf to a point $\iota: \{ * \} \to \Qc$, so that
\[\rk p_* (E \otimes q^* w) = \rk
\iota^* p_* (E \otimes q^* w) = \chi
\left( \iota^* p_* (E \otimes q^* w) \right).\]
Let $X$ denote the K3 surface that corresponds to the chosen point $\iota$ in $\Qc$, then we have the following pullback diagram:
\begin{diagram}
\Yc & \lTo^q    & \Yc \fprod{\Xc} \Qc
    & \lTo^\kappa   &
    X
    \\
\dTo && \dTo^p && \dTo_{\gamma}
    \\
\Xc & \lTo          & \Qc
    & \lTo^\iota    &
    \{ * \}
    \\
\end{diagram}

Now we can compute the rank. Note that we use base change formula in the first line and orthogonality of $v$ and $w$ in the second line:
\[
\rk p_* (E \otimes q^* w) =
    \chi \left( \iota^* p_* (E \otimes q^* w) \right) =
    \chi \left( \gamma_* \kappa^* (E \otimes q^* w) \right) = \chi \left(
    \kappa^* E \otimes (q\kappa)^* w \right) = 0.
\]

\textit{Step 2: ``descent data'' for $p_* (E \otimes q^* w)$ and $L$.}
Let $a: A \to \Qc$ be a smooth atlas. Then its composition $\nu a$ with $\nu : \Qc \to \Nc_v$ is a smooth atlas for $\Nc_v$, since formal smoothness can be verified by a lifting property and finite presentation is automatic.
Introduce projection morphisms $q_1$, $q_2$, $r_1$, $r_2$, summarized in the diagram below, where
$B = A \fprod{\Qc} A$ and
$C = A \fprod{\Nc_v} A$:
\begin{diagram}
\Qc
    & \lTo^{a}    & A
    & \lTo^{q_i} & B
    \\
\dTo_\nu && \dTo_{=}
    && \dTo_{\pi}
    \\
\Nc_v
    & \lTo^{\nu a}    & A
    & \lTo^{r_i} & C
    & \lTo^{r_{ij}} & C\fprod{A}C
    \\
\end{diagram}
We let $r_{12}$, $r_{23}$, $r_{13}$ be projection and composition morphisms from $C\fprod{A}C$ to $C$ that determine the structure of a groupoid.
Since fibers of $\nu$ are $B\Gm$, we get, by the magic square diagram, that $\pi$ is a $\Gm$-fibration given by some line bundle $T$. We know that the complex $p_* (E \otimes q^* w)$ on $\Qc$ corresponds to a complex $F$ on $A$ that has a gluing isomorphism $q_1^* F \to q_2^* F$ on $B$. Since $q_i = r_i \pi$ and by Lemma \ref{Lem:twisted_descent_datum}, if $F$ was indecomposable, we would get an isomorphism
$\psi: r_1^* F \to r_2^* F \otimes T^{\otimes k}$ for some integer $k$.
The complex $F$ is not necessarily indecomposable, so we wish to apply Lemma \ref{Lem:twisted_descent_datum} to each summand.
However, since $p_* (E \otimes q^* w)$ is a complex of sheaves on a stack with $B \Gm$ stabilizers, we can calculate the weight of the $\Gm$-action on the fibers which would determine the corresponding twist, and we will conclude that the twist is the same for each summand of $F$. Similar to Step 1, let $\iota : B\Gm \to \Qc$ be an embedding of a point with its stabilizer, then we have the following commutative diagram:
\begin{diagram}
\Yc & \lTo^q    & \Yc \fprod{\Xc} \Qc
    & \lTo^\kappa   &
    X \times B\Gm
    \\
\dTo && \dTo^p && \dTo_{\gamma}
    \\
\Xc & \lTo          & \Qc
    & \lTo^\iota    &
    B\Gm
    \\
\end{diagram}
Then consider the restriction along $\iota$, where $E_X$ and $w_X$ denote the restrictions of $E$ and $w$ to $X$:
$$ \iota^* p_* (E \otimes q^* w) =
\gamma_* \left( E_X \otimes w_X
\right) 
.$$
We can see that $\Gm$ acts on $w_X$ trivially and on $E_X$ by tautological scaling, so the resulting action on the cohomology of $E_X \otimes w_X$ is also tautological scaling with weight one. Therefore, we have an isomorphism
$$\psi: r_1^* F \to r_2^* F \otimes T
.$$

Since $q_1^* F \to q_2^* F$ satisfies the cocycle condition, we get that the following composition, denoted by $1\otimes f$, is an isomorphism:
$$  1\otimes f = \left( r_{13}^* \psi \right)^{-1}
    \circ r_{23}^* \psi
    \circ r_{12}^* \psi :
    r_{13}^*r_1^* F \otimes T \to 
    r_{13}^*r_1^* F \otimes T^{\otimes 2}
\mbox{.} $$

\textit{Step 3: cocycle condition for $\varphi \deq \det \psi$.}
Let us first write $\varphi$, remembering from Step 1 that rank is zero:
$$ \varphi : \det r_1^* F \to 
\det r_2^* F \otimes T^{\rk F} =
\det r_2^* F
.$$
So we have a ``descent datum'' for $\det F$, and now we verify that the cocycle condition holds:
\begin{equation*}
\begin{split}
\left( r_{13}^* \varphi \right)^{-1}
    \circ r_{23}^* \varphi
    \circ r_{12}^* \varphi
    =
    \det (1 \otimes f)
    = 1\otimes f^{\rk F} = 1
\mbox{.}
\end{split}
\end{equation*}
So $L$ satisfies the cocycle condition and hence descends to $\Nc_v$.
\end{proof}

Recall that $\Mcal_v \to \Kc$ is the relative moduli scheme of stable sheaves over the stack of quasi-polarized K3 surfaces, while $\Nc_v = \Mcal_v \fprod{\Kc} \Xc \to \Xc$ is the same over pointed quasi-polarized surfaces, so every fiber of $\Nc_v \to \Mcal_v$ is naturally the underlying surface.
Let $L_w$ now denote the line bundle on $\Nc_v$ constructed in Lemma \ref{Lem:theta_on_ptd}.
We now want to argue that $L_w$, possibly up to a twist by the quasi-polarization, is isomorphic to the pullback along $\Nc_v \to \Mcal_v$ of some line bundle on $\Mcal_v$.

\begin{Lem}[Marian-Oprea \cite{MO}]
    Pick two orthogonal K-theory vectors:
    $$v = r_v \Oc + d_v \Hc + a_v \Oc_\sigma ,$$
    $$w = r_w \Oc + d_w \Hc + a_w \Oc_\sigma $$
    in the algebraic K-theory $\K \Yc$, where we recall that $\Yc$ is the universal pointed K3 surface with the universal quasi-polarization $\Hc$, and we use $\sigma$ to denote the class of the natural section $\Xc \to \Yc$. Let $L$ be the line bundle that we descended from $\det$ $q_i^* p_* (E \otimes q^* w)$ on $\Qc$ to $\Nc_v$. Then the restriction of $L$ to a fiber $X$ of $\Nc_v \to \Mcal_v$ is isomorphic to a power of the quasi-polarization $H^n = \Hc^n_{|X}$, and $n$ is independent of the choice of a fiber.
\end{Lem}
\begin{proof}
    See the discussion above Equation (4.1) on Page 2080 of the paper ``On Verlinde sheaves and strange duality'' by Marian and Oprea \cite{MO}.
\end{proof}

This lemma shows that $L_w$ and a tensor power of the polarization $\Hc^n$ are fiberwise isomorphic, and therefore
the twist $L_w \otimes \Hc^{-n}$
of the determinant line bundle
$L_w$ on $\Nc_v$
comes as a pullback from $\Mcal_v$.
Let us denote a suitable line bundle on $\Mcal_v$ by $\Theta_w$.

\begin{Def}
    Pick two orthogonal algebraic K-theory classes $v$ and $w$ over
    $\Yc = \Xc \fprod{\Kc} \Xc$.
    There exists a line bundle $\Theta_w$ on $\Mcal_v$ whose pullback to
    $\Nc_v = \Mcal_v \fprod{\Kc} \Xc$ is isomorphic, up to a twist by the universal quasi-polarization, to the determinant line bundle $L_w$. This line bundle $\Theta_w$ is called a \emph{theta line bundle}.
\end{Def}

\subsection{Constructing the Strange Duality morphism}

Recall that our aim is to extend the definition of the Strange Duality morphism to the relative case. Pointwise, the morphism is expected to establish a duality between two vector spaces of global sections. The relative version of cohomology is the derived pushforward functor, therefore we will work with the pushforwards of the theta line bundles.

\subsubsection*{Assumptions}
Recall that $\Kc$ stands for the moduli stack of quasi-polarized K3 surfaces and $\Xc \to \Kc$ denotes the universal K3 surface with quasi-polarization $\Hc$.
Let $v$ and $w$ in $\K \Xc$ be two poinwise orthogonal numerical characteristics, that is $\chi(v\otimes w)=0$ on each K3 surface in the family.
Assume that pointwise on $\Kc$,
semistable sheaves of classes $v$ and $w$ are stable.
By the results of the previous section,
this ensures that we have relative moduli spaces $\pi_v: \Mcal_v \to \Kc$ and $\pi_w: \Mcal_w \to \Kc$ with the theta line bundles $\Theta_w$ on $\Mcal_v$ and $\Theta_v$ on $\Mcal_w$. Let $\pi: \Mcal_v \fprod{\Kc} \Mcal_w \to \Kc$ denote the natural projection.

\begin{Def}
The pushforwards $W \deq \pi_{v*} \Theta_w$ and $V \deq \pi_{w*} \Theta_v$ are known as the \emph{Verlinde complexes}.
\end{Def}

\begin{Def}
Define the \emph{Brill-Noether locus} $\Theta$ of jumping zeroth cohomology on $\Mcal_v \fprod{\Kc} \Mcal_w$ as follows:
$$  \Theta \deq \left\{
    (X,E,F) \mid
    \RGamma^0 (X,E \otimes F) \neq 0
\right\}
\mbox{.} $$
\end{Def}

One naturally expects $\Theta$ to be a divisor or coincide with the whole locus $\Mcal_v(X) \times \Mcal_w(X)$ over each point $[X] \in \Kc$. The locus in $\Kc$ where $\Theta$ is not a divisor is of codimension at least two if the complement is not empty.

\begin{Lem}[\emph{cf.} \protect{\cite[Remark 4.2]{MO}}]
There exists a line bundle $\Tc$ on $\Kc$ so that we have an isomorphism on $\Mcal_v \fprod{\Kc} \Mcal_w$:
$$  \pi^* \Tc \otimes \Oc(\Theta) \cong
\Theta_w \boxtimes \Theta_v
\mbox{.} $$
\end{Lem}

We can pushforward the isomorphism to $\Kc$. After using projection formula twice as well as flat base change isomorphism, we get the following:
$$  \Tc \otimes \pi_* \Oc(\Theta) \cong
\pi_* \left( \Theta_w \boxtimes \Theta_v \right) \cong W \otimes V
\mbox{.} $$

The section of $\Oc(\Theta)$ corresponds to a section $\pi_* \Oc(\Theta)$, so by local triviality of $\Tc$, it corresponds locally to a morphism $W^\vee \to V$. We will denote this morphism by $\SD$ and call it the Strange Duality morphism, remembering that it is only defined up to the twist $\Tc$:
\begin{equation}
\label{Eq:SD_mm}
    \SD : W^\vee \to V
    \mbox{.}
\end{equation}


\ifx\undefined\bysame
\newcommand{\bysame}{\leavevmode\hbox to3em{\hrulefill}\,}
\fi

\end{document}